\documentclass[12pt,twoside]{amsart}
\usepackage{amssymb,amsmath,amsthm}
\usepackage{verbatim}
\usepackage{graphicx}
\usepackage{epsfig}
\usepackage{color, soul}
\usepackage[all]{xy}

\newtheorem{theorem}{Theorem}[section]

\newtheorem{definition}[theorem]{Definition}
\newtheorem{corollary}[theorem]{Corollary}
\newtheorem{proposition}[theorem]{Proposition}
\newtheorem{remark}[theorem]{Remark}
\newtheorem{example}[theorem]{Example}

\def\gen{\mbox{gen}}

\def\trin{\vert\vert\vert}

\title[The reconstruction formula and duality]{The reconstruction formula for Banach frames and duality}

\author{Daniel Carando, Silvia Lassalle and Pablo Schmidberg}

\thanks{This project was supported in part by UBACyT X038, UBACyT X218 and CONICET}

\address{Departamento de Matem\'{a}tica - Pab I,
Facultad de Cs. Exactas y Naturales, Universidad de Buenos Aires,
(1428) Buenos Aires, Argentina}

\email{dcarando@dm.uba.ar, slassall@dm.uba.ar, pschmid@dm.uba.ar}


\begin{document}
\parindent 10pt
\baselineskip=.70cm

\begin{abstract}
We study conditions on a Banach frame that ensures the validity of a reconstruction formula. In particular, we show that any Banach frames for (a subspace of) $L_p$ or $L_{p,q}$ ($1\le p < \infty$) with respect to a solid sequence space always satisfies an unconditional reconstruction formula. The existence of reconstruction formulae allows us to prove some James-type results for atomic decompositions:  an unconditional atomic decomposition (or unconditional Schauder frame) for $X$ is shrinking (respectively, boundedly complete) if and only if $X$ does not contain an isomorphic copy of $\ell_1$ (respectively, $c_0$).
\end{abstract}

\maketitle
\section*{Introduction}

Banach frames emerged in the theory of frames related to Gabor
and Wavelet analysis and were formally introduced in 1991 by Gr\"ochenig
\cite{g} as an extension of the notion of frames for Hilbert spaces
to the Banach space setting. Before the concept of Banach frames was formalized,
it appeared in the foundational work of Feichtinger and Gr\"ochenig \cite{FG1, FG2} related to atomic decompositions. Loosely speaking, atomic decompositions allow a representation of every element of the space via a series expansion in terms of a fixed sequence of elements, {\it the atoms}. Banach frames, on the other hand, ensure reconstruction via a bounded {\it synthesis operator} and, many times, to find an explicit formula for this operator presents additional difficulties. One of our main results (Theorem \ref{teo: BF lattice no c0 es DA incond}) shows that the synthesis operator associated to a wide class of Banach frames, is given by a series expansion with unconditional convergence, whose coefficients depend linearly and continuously on the entry.

Banach frames and atomic decompositions appeared in the field of applied mathematics providing applications to signal processing, image processing and sampling theory among other areas. In the wavelet context, Frazier and  Jawerth presented decompositions for Besov spaces in their early work \cite{FrJa1}, and later for distribution spaces in \cite{FrJa2}, where a new approach to the traditional atomic decomposition of Hardy spaces can be found. Feichtinger characterized Gabor atomic decomposition for modulation spaces \cite{Fei} and, at the same time, the general theory was developed in his joint work with Gr\"ochenig \cite{FG1, FG2}.  Here, the authors show that reconstruction through atomic decompositions are not limited to Hilbert spaces. Indeed, they construct frames for a large class of Banach spaces, namely the {\it coorbit spaces}. Thereafter, a vast literature was dedicated to the subject (see the references in~\cite{CaHaLa}).

We focus our discussion within the framework of abstract approximation theory in Banach spaces. This allows us to relate the concepts of Banach frames and atomic decomposition to properties of Banach spaces such as separability and reflexivity.

We show that a Banach frame for a Banach space $X$ with
respect to a solid space $Z$ (in our terminology, an unconditional Banach frame) admits a reconstruction formula
whenever $X$ does not contain a copy of $c_0$. In this case, the
Banach frame automatically defines an unconditional atomic
decomposition. This holds for reflexive Banach spaces or spaces with
finite cotype. As a consequence, any Banach frame for $L_p$ ($1\le
p<\infty$) and Lorentz function space $L_{p,q}$ ($1< p,
q<\infty$), or any of their  subspaces, with respect to a solid sequence space admits a reconstruction formula.
The reconstruction formula for Banach frames is applied to obtain some
James-type results: an unconditional
atomic decomposition or Schauder frame for $X$ is shrinking if and only if $X$ does
not contain a copy of $\ell_1$, and is boundedly complete if and
only if $X$ does not contain a copy of $c_0$. This improves some
results of~\cite{CaLa} and~\cite{Liu}.

The paper is organized as follows. In the first section,
we introduce the basic definitions that will be used throughout.
In section 2, we recall the definitions of shrinking and boundedly complete
atomic decompositions, and present some basic duality results. Section 3 is
devoted to the main results of the article.

For further information on atomic decompositions and Banach frames
see, for example, \cite{CaHaLa} and the references therein. For an historical survey on some aspects of frame theory for Hilbert spaces see \cite{Heil} and the references therein. We refer to \cite{Di, LiTzI, LiTzII} for a background in Banach spaces and Banach lattices.

\section{Background and generalities}
By a Banach sequence space we mean a Banach space of scalar
sequences, indexed by $\mathbb N$, for which the coordinate functionals are continuous.
We say that the space is a Schauder sequence space if, in
addition, the unit vectors $\{e_i\}$ given by $(e_i)_j=\delta_{i,j}$
form a basis for it. In this case, a sequence $a=(a_i)$ can be
written as $a=\sum_{i=1}^\infty a_ie_i$.

We start by recalling the definition of a Banach frame:
\begin{definition}\label{def: BF}
Let $X$ be a Banach space and $Z$ be a Banach  sequence space. Let
$(x'_i)$ be a sequence in $X'$ and let $S\colon
Z\to X$ be a continuous operator. The pair $((x'_i),S)$ is said to be a Banach frame for
$X$ with respect to $Z$ if for all $x\in X$:

{\rm (a)} $(\langle x'_i,x\rangle) \in Z$,

{\rm (b)} $A\|x\|\le \| (\langle x'_i,x\rangle)  \|_{Z}\le B \|x\|$,
with $A$ and $B$ positive constants,

{\rm (c)} $x=S(\langle x'_i,x\rangle)$.
\end{definition}

The operator $S$ is said to be the \textit{synthesis operator}.
Conditions (a) and (b) allow the definition of the
\textit{analysis operator} $J\colon X\to Z, \quad Jx\colon =(\langle
x'_i,x\rangle)_i$. The synthesis and analysis operators determine
the Banach frame in the following sense: if $((x'_i),S)$ is a Banach
frame then,  $SJ=id_X$ and $x'_i=J'e'_i$. On the other hand, if
$S\colon Z\to X$ and $J\colon X\to Z$ are continuous operators such
that $SJ=id_X$ then, $((J'e'_i),S)$ is a Banach frame for $X$ with
respect to $Z$ (see~\cite[Remark 1.2]{CaLa}
and~\cite[Page 712]{CaChSt}).

Whenever $Z$ is a Schauder sequence space, the continuity of $S$
gives the reconstruction formula for the Banach frame:
\begin{equation}\label{eq:reconstruction formula}
x=\sum_{i=1}^\infty \langle x'_i,x\rangle Se_i
\end{equation}
for all $x\in X$. If the canonical sequence $(e_i)$ does not span $Z$,
the reconstruction formula does not necessarily hold, even for separable Banach spaces $X$.
Let us see a simple example of this:
\begin{example}\label{ej: BF no DA c0} \rm
Let $X=c_0$ (the space of null sequences) and $Z=c$ (the space of convergent sequences).
We consider the following operators:
$$
S\colon Z\to X, \quad Sa\colon =(\ell,a_1-\ell,a_2-\ell,\ldots),\
\mbox{where}\ \ell=\lim_i a_i,
$$
and
$$
J\colon X\to Z, \quad Jx\colon =(x_1,x_2+x_1,x_3+x_1,\ldots).
$$
Note that $S$ and $J$ are bounded with $\|S\|=\|J\|=2$. If we set $x'_1\colon =e'_1$ and $x'_i\colon =e'_1+e'_i$ for
$i\geq 2$, then  $Jx=(\langle
x'_i,x\rangle)$ and $SJ=id_X$, so $((x'_i),S)$ is a Banach frame for $c_0$ with respect to $c$.
Let us see that the reconstruction formula does not hold for $x=e_1$. Since $Se_i=e_{i+1}$, we have for each $n$
$$
\begin{aligned}
\sum_{i=1}^n \langle x'_i,e_1\rangle Se_i & = \langle
x'_1,e_1\rangle e_2+\sum_{i=2}^n \langle
x'_i,e_1\rangle e_{i+1}\\
& = e_2+e_3+\cdots +e_n.
\end{aligned}
$$
Then, $\sum_{i=1}^\infty \langle x'_i,e_1\rangle Se_i$ does not converge.
\end{example}
\bigskip

One of the purposes of this work is to establish conditions that ensures that
the reconstruction formula is satisfied by a Banach frame.
A similar but subtly different structure is that of atomic decomposition.
The reconstruction formula is imposed as part of the definition, and in return
we give up the existence of a linear operator $S$ defined on the whole space $Z$:

\begin{definition}\label{def: DA}
Let $X$ be a Banach space and $Z$ be a Banach  sequence space. Let
$(x'_i)$ and $(x_i)$ be sequences in $X'$ and $X$ respectively. We
say that $((x_i'),(x_i))$ is an atomic decomposition of $X$ with
respect to $Z$ if for all $x\in X$:

{\rm (a)} $(\langle x'_i,x\rangle) \in Z$,

{\rm (b)} $A\|x\|\le \| (\langle x'_i,x\rangle)  \|_{Z}\le B \|x\|$,
with $A$ and $B$ positive constants,

{\rm (c)} $x=\sum_{i=1}^\infty \langle x'_i,x\rangle x_i$.
\end{definition}

The comments above say that a Banach frame with respect to a Schauder
sequence space automatically defines an atomic decomposition. Moreover,
any Banach frame satisfying a reconstruction formula defines an atomic decomposition.

Let us describe a sort of converse of this statement. A separable
Banach space admits an atomic decomposition if an only if it has the
bounded approximation property  (see \cite{JoRoZi, Pe} and also
\cite[Theorem 2.10]{CaHaLa}). Moreover, if $((x_i'),(x_i))$ is an
atomic decomposition of $X$ with respect to some Banach sequence
space $Z$, it is always possible to find a Schauder sequence space
$X_d$ and an operator $S\colon X_d \to X$ such that $Se_i=x_i$ and
$((x_i'),(x_i))$ is also an atomic decomposition of $X$ with respect
to $X_d$. In this case, $((x_i'),S)$ turns out to be a Banach frame
for $X$ with respect to $X_d$. Therefore, we might say that an
atomic decomposition defines a Banach frame, as long as we are
allowed to change the sequence space involved. Note that the natural
inclusion from $c_0$ into $\ell_\infty$ defines an atomic
decomposition for $c_0$ with respect to $\ell_\infty$, but there is
no Banach frame for $c_0$ with respect to $\ell_\infty$. Therefore,
it is sometimes really necessary to change the sequence space.

On the other hand, the Banach frame of Example~\ref{ej: BF no DA c0}
does not define an atomic decomposition, since the reconstruction
formula does not hold.

\medskip
Let $((x_i'),(x_i))$ be an atomic decomposition for
$X$ with respect to a Banach sequence space $Z$. There is a natural
procedure that allows us to replace  $Z$ by a Schauder sequence
space $X_d$ so that $((x_i'),(x_i))$ is also an atomic decomposition
of $X$ with respect to $X_d$ (see \cite[Theorem 2.6]{CaHaLa}). For the sake of completeness, we sketch the construction of $X_d$ under the assumption that $x_i$ is nonzero, for all $i$, since this assumption avoids some technicalities. Consider $c_{00}$ the
linear space of scalar finite support sequences with unit vectors
$(e_i)$ endowed with the norm:
$$
\|\sum_i a_ie_i\|=\sup_n\|\sum_{i=1}^n a_ix_i\|_X.
$$
Now, define $X_d$ as the completion of $c_{00}$ with the norm given
above. In fact,
\begin{equation}\label{canonical seq sp}X_d = \big\{(a_i)\big/\ \sum_{i=1}^\infty a_i x_i\
\mbox{converges}\big\}
\end{equation} and $((x_i'),(x_i))$ turns out to be an atomic decomposition of $X$ with respect to $X_d$.
We will call this Schauder sequence space {\it
the canonical associated Schauder space} to the corresponding atomic
decomposition. We also remark that Theorem 2.6 of  \cite{CaHaLa} (or the existence of $X_d$) implies  that a Banach space admits an atomic decomposition if and only if it is
complemented in a Banach space with an basis.

One of the advantages of working with Banach frames or atomic decomposition is that
these structures have a nicer behavior than that of basis with respect to subspaces.
First, note that
if $((x'_i),S)$ is a Banach frame for  $X$ with respect to $Z$ then $X$ is isomorphic
to a complemented subspace of $Z$. This property relies on the simple fact that $SJ=id_X$
and therefore $JS$ is the desired projection. We also have:

\begin{remark}\label{rem: subsp and proj}
Let $X$ be a Banach space and $Z$ be a Banach sequence space. Suppose $(x'_i)\subset X'$
satisfies properties {\rm (a)} and  {\rm (b)} of Definition~\ref{def: DA} and let $P\colon X\to X$
be a projection then,
\begin{enumerate}
\item[(i)]
if $((x'_i),S)$ is a Banach frame for $X$ with respect to $Z$ then, $((P'x'_i),PS)$ is a Banach
frame for the space $PX$ with respect to $Z$.
\item[(ii)]
if $((x'_i),(x_i))$ is an atomic decomposition for $X$ with respect to $Z$ then, $((P'x'_i),(Px_i))$
is an atomic decomposition for the space $PX$ with respect to $Z$.
\end{enumerate}
\end{remark}

\begin{proof}
Let $((x'_i),S)$ be a Banach frame. Since $\langle
P'x'_i,x\rangle=\langle x'_i,Px\rangle=\langle x'_i,x\rangle$ for all $x\in PX$ and for all $i$,
the sequence $(\langle P'x'_i,x\rangle)$ belongs to $Z$, for all $x\in PX$.
\medskip

Also, if $x\in PX$, we have that $\|(\langle
P'x'_i,x\rangle)\|_Z=\|(\langle x'_i,x\rangle)\|_Z$ and
 $S(\langle P'x'_i,x\rangle)=S(\langle x'_i,x\rangle)=x$. Thus, statement (i) is proved.
Now, take  $((x'_i),(x_i))$ an atomic decomposition. To prove (ii) it only remains to show
that the reconstruction formula holds. Indeed, we have
$$
Px=P(Px)=P\big(\sum_{i=1}^\infty \langle x'_i,Px\rangle x_i\big)= \sum_{i=1}^\infty
\langle P'x'_i,Px\rangle Px_i.
$$\end{proof}
\medskip

In most applications, sequence spaces associated to a Banach frame are solid.
Recall that a Banach sequence space $Z$ is called solid if for any pair of sequences
$a=(a_i)$ and $b=(b_i)$ with $a\in Z$ and such that  $|b_i| \le |a_i|$ for all $i$,
we have that $b\in Z$ and $\|b\|\le\|a\|$. Classical examples of solid sequence spaces are $c_0$, $\ell_p$ and Lorentz and Orlicz sequence spaces. Solid sequence spaces are Banach lattices modeled over the natural numbers, with the coordinatewise  order, and are also called K\"{o}the sequence spaces, or normal Banach sequence spaces.
For the theory of Banach lattices we refer to \cite{AlBu} and \cite{LiTzII}. Any Banach space with a 1-unconditional basis is a solid space and those with arbitrary unconditional basis can be renormed to be solid. On the other hand, if $Z$ is a solid sequence space, then the canonical unit vectors form a 1-unconditional basic sequence.  This motivates the
following:

\begin{definition}
A Banach frame with respect to a solid sequence space is said to be an unconditional Banach frame.
\end{definition}

Note that in our definition of unconditional Banach frame we do not require the solid sequence space to have a basis.
Unconditional Banach frames have a natural counterpart in the atomic decomposition framework: unconditional atomic decompositions. This last concept introduced and studied
in~\cite{CaLa} is equivalent to that of ``framing for Banach
spaces'' given in~\cite{CaHaLa} and of ``unconditional frame'' given in~\cite{CaDiOdSchZs,Liu} (see the comments at the end of this section).

\begin{definition}\label{def:desc at incond}
An atomic decomposition $((x'_i), (x_i))$ for $X$ with respect to a Banach sequence space $Z$ is said to be unconditional if for any $x\in X$,
its series expansion $\sum_{i=1}^\infty \langle x_i', x\rangle x_i$ converges unconditionally, that is
$$
x=\sum_i \langle x'_i,x\rangle x_i,
$$
with unconditional convergence.
\end{definition}

It is known that if a series $\sum_i x_i$ converges
unconditionally, then, for every bounded sequence of scalars
$\{a_i\}$, the series $\sum_i a_i x_i$ converges and the operator
$\ell_\infty\to X$ defined by $(a_i)\mapsto \sum_i a_i x_i$ is a
bounded linear operator (see~\cite{LiTzI}, page 16). Thus, a
repeated use of the uniform boundedness principle (or a single
application of the bilinear Banach-Steinhaus theorem~\cite[Ex 1.11(a)]{DeFl}) shows:

\begin{remark}\label{obs: DA incond y bilineal}
Let $((x'_i),(x_i))$ be an unconditional atomic decomposition for
$X$. Then the mapping
$$
B\colon X\times \ell_\infty\to X,\quad B(x,a)\colon= \sum_i a_i
\langle x'_i,x\rangle x_i
$$
is a well defined bounded bilinear operator.
\end{remark}

The norm of the bilinear operator $B$ defined above can be seen as an unconditional constant for the atomic decomposition $((x'_i),(x_i))$. Equivalent constants have been introduced in \cite{CaHaLa} and \cite{CaDiOdSchZs}.

If $((x'_i),(x_i))$ is an unconditional atomic
decomposition for $X$, it is always possible to find a solid sequence space with Schauder basis $X_d$ and an operator $S\colon X_d \to X$
such that $Se_i=x_i$ and $((x_i'),(x_i))$ is also an unconditional
atomic decomposition of $X$ with respect to $X_d$. The construction
is similar to (\ref{canonical seq sp}) and can be found in~\cite[Theorem~3.6]{CaHaLa}. Assuming again that $x_i\neq 0$ for all $i$,
the solid sequence space with Schauder basis $X_d$ is
\begin{equation}\label{canonical solid seq sp}
X_d\colon = \big\{(a_i)\big/\ \sum_i a_i x_i\ \mbox{converges
unconditionally}\big\},
\end{equation}
endowed with the norm $\|(a_i)\|{_{X_d}}\colon
=\sup_{b\in B_{\ell_{\infty}}}\|\sum_i b_i a_i x_i\|$. We will refer to this space as {\it the canonical  solid Schauder space} associated to the corresponding atomic decomposition. Therefore, an unconditional
atomic decomposition defines a Banach frame with respect to some solid Schauder sequence space. Conversely, a Banach frame with respect to a solid Schauder sequence space defines an unconditional atomic decomposition. Moreover, \cite[Theorem 3.6]{CaHaLa} says that a Banach space admits an unconditional atomic decomposition if and only if it is complemented in a Banach space with an unconditional basis.
\medskip

Most of the properties of atomic decompositions we will study are independent of the associated sequence space. Also, the construction of the canonical Schauder spaces (\ref{canonical seq sp}) and (\ref{canonical solid seq sp}) associated to an atomic decomposition, \textit{only} involve
the reconstruction formulae and not the original sequence space. Therefore, unless specific properties of the associated space are required, we will talk about atomic decompositions without reference to any sequence space (having in mind, if necessary, the canonical sequence spaces associated to the decomposition). When the Banach sequence space is disregarded, the concept of (unconditional) atomic decomposition is equivalent  to that of (unconditional) Schauder frame in the sense of~\cite{CaDiOdSchZs}.

\section{Some remarks on duality for atomic decompositions}

In order to relate atomic decomposition to duality properties of Banach spaces, the notion of shrinking atomic decomposition was introduced in~\cite{CaLa}.
Before we recall the definition, consider the linear
operators $T_N\colon X \to X$ by $T_N(x)=\sum_{i\ge N} \langle x'_i,
x\rangle x_i$. These operators are uniformly bounded by the uniform
boundedness principle. With this notation we have:
\begin{definition}\label{def: DA shrinking}
Let $((x_i'),(x_i))$ be an atomic decomposition of
$X$. We say that
$((x_i'),(x_i))$ is shrinking if for all $x'\in X'$
\begin{equation}\label{condicion shrinking}
\|x'\circ T_N\| \longrightarrow 0.
\end{equation}
\end{definition}
As a matter of fact, the definition in \cite{CaLa} was referred to as an atomic decomposition with respect to a concrete Banach sequence space $Z$. However, it must be noted that the condition (\ref{condicion shrinking}) is independent of the associated sequence space.
In general, $T_N\circ T_M$ may be different from $T_{\min(N,M)}$,
which means that $T_N$ in not a projection on the closure of
$[x_i\colon i\ge N]$. This shows one of the main differences between
atomic decompositions and bases. Indeed, the definition above
is not equivalent to $\|x'|_{[x_i\colon i\ge N]}\|$ going to 0 for
every $x'\in X'$ (see \cite{CaLa} for details). For some particular
atomic decompositions, which in fact are simultaneously Banach
frames, it is shown in \cite[Theorem 1.4]{CaLa} that shrinking
atomic decompositions behave as shrinking Schauder bases in the
following sense: suppose $X_d$ is a Schauder sequence space with
basis $(e_i)$ and there exists a synthesis operator $S\colon X_d \to
X$ such that $Se_i=x_i$ and $((x_i'),(Se_i))$ is an atomic
decomposition of $X$ with respect to $X_d$. Then the the dual pair
$((Se_i), (x'_i))$ is an atomic decomposition for $X'$ with respect
to $(X_d)'$ if and only if $((x'_i), (Se_i))$ is shrinking. Let us
see that we can extend this result to arbitrary atomic
decompositions:

\begin{proposition}\label{prop: DA shr y serie conv}
Let $X$ be a Banach space and  $((x'_i),(x_i))$ be an atomic decomposition for $X$.
The following are equivalent:
\begin{enumerate}
\item[(i)] the pair $((x'_i),(x_i))$ is shrinking,
\item[(ii)] the pair $((x_i),(x'_i))$ is an atomic decomposition for $X'$,
\item[(iii)] for all $x'\in X'$,  $\sum_{i=1}^\infty \langle x',x_i\rangle x'_i$ is convergent.
\end{enumerate}
\end{proposition}
\begin{proof} For (i) $\Rightarrow$ (ii), consider $X_d$ the canonical Schauder space associated
to $((x'_i),(x_i))$ presented in~(\ref{canonical seq sp}). As we have mentioned,  $((x'_i),(x_i))$ is
shrinking as an atomic decomposition with respect to $X_d$, since the definition of
shrinking atomic decomposition is independent of the sequence
space. Then the result follows from \cite[Theorem~1.4]{CaLa}. The
implication (ii) $\Rightarrow$ (iii) is immediate and (iii)
$\Rightarrow$ (i) follows directly from the equality $\|x' \circ
T_N\|=\big\|\sum_{i=N}^\infty \langle x',x_i\rangle x'_i\big\|$.
\end{proof}

As a consequence of this result, a Banach space admitting a
shrinking atomic decomposition has necessarily a separable dual. In
particular, $\ell_1$ does not admit such a decomposition. Also,
Proposition~\ref{prop: DA shr y serie conv} shows the equivalence
between the notion of shrinking atomic decomposition and the concept
of pre shrinking atomic decomposition given in~\cite{Liu}.
\medskip

Another concept related to duality is that of
boundedly complete atomic decomposition, which was introduced in \cite{CaLa}
and is a natural extension of the definition of boundedly complete Schauder basis (in~\cite{Liu}, this last concept is defined as
``pre boundedly complete'').

\begin{definition}\label{def: DA acot completa}
Let $X$ be a Banach space and let $((x_i'),(x_i))$ be an atomic decomposition of
$X$. The atomic decomposition
is said to be boundedly complete if for each $x''\in X''$, the
series $\sum_{i=1}^\infty \langle x'', x'_i\rangle x_i$ converges in
$X$.
\end{definition}

We have already mentioned that ``admitting an atomic decomposition''
is a property that is inherited by complemented subspaces
(Remark~\ref{rem: subsp and proj}). The same happens if we require
the atomic decomposition to be shrinking or boundedly complete, this fact will be used later:

\begin{remark}\label{obs: DA acot completa proy}
Let $((x'_i),(x_i))$ be an atomic decomposition of $X$ and let $P\colon X\to X$ be a continuous linear projection.
If $((x'_i),(x_i))$ is boundedly complete (shrinking), then $((P'x'_i),(Px_i))$ is also a boundedly complete (shrinking) atomic decomposition for $PX$.
\end{remark}

\begin{proof}
Put $Y=PX$ and let us show the statement for complete boundedness. Given $y''\in Y''$, consider $\sum_{i=1}^\infty \langle P''y'',x'_i\rangle x_i$ which converges since  $((x'_i),(x_i))$ is boundedly complete. Then,
$\sum_{i=1}^\infty \langle y'',P'x'_i\rangle Px_i=P\big(\sum_{i=1}^\infty\langle
P''y'',x'_i\rangle x_i\big)$ is also convergent. The arguments for shrinking atomic decompositions are similar.
\end{proof}

If an atomic decomposition is modeled on a Schauder sequence space
$X_d$ with a boundedly complete basis, then the atomic decomposition
is boundedly complete. It is easy to find an example to show that
the converse of this result is false. Indeed, take $X$ a reflexive
Banach space with basis $(f_i)$. Consider $X_d=X\oplus\ c_0$ with
the basis which alternates the elements
$f_i$ with the elements of the canonical basis of $c_0$. The natural
inclusion $J:X\hookrightarrow X_d$ and projection $S:X_d\to X$
define a boundedly complete atomic decomposition, but clearly the
basis $(e_i)$ of $X_d$ is not boundedly complete.

The following remark shows that not every separable Banach space admits a boundedly complete atomic decomposition (take, for instance, $X=c_0$).

\begin{remark}\label{obs: DA acot compl y bidual}
Let be $X$ a Banach space with a boundedly complete atomic decomposition. Then, $X$ is complemented in its bidual $X''$.
\end{remark}

\begin{proof}
Let $((x'_i),(x_i))$ be a boundedly complete atomic
decomposition for $X$. By the Banach-Steinhaus theorem, the
following mapping is well defined and bounded:
$$
P\colon X''\to X,\quad Px''\colon =\sum_{i=1}^\infty \langle x'',x'_i\rangle
x_i.
$$
Now, the reconstruction formula says that $P$ is the desired projection.
\end{proof}

A kind of converse of the previous result holds for unconditional
atomic decompositions (see Corollary~\ref{coro: DA acot compl y
bidual 2}).

\section{The reconstruction formula and the James-type results}

The following result provides us with a sufficient condition to ensure reconstruction formulae for unconditional Banach frames. The proof is based on that of a theorem of Fiegel, Johnson and Tzafriri~\cite[Proposition 1.c.6]{LiTzII} for Banach lattices.

\begin{theorem}\label{teo: BF lattice no c0 es DA incond}
Let $((x'_i),S)$ be an unconditional Banach frame for  $X$. If $X$ does not contain an isomorphic copy of $c_0$, then we have reconstruction formula for $((x'_i),S)$. More precisely, if $(e_i)$ denotes the sequence of the canonical unit vectors of the solid sequence space $Z$ associated to the frame, then for all $x\in X$ we have $$
x=\sum_i \langle x'_i,x\rangle Se_i
$$
unconditionally. Equivalently,  $((x'_i),(Se_i))$ is an unconditional atomic decomposition for $X$.
\end{theorem}

\begin{proof} Taking a subspace of $Z$ if necessary, we may assume that $Se_i\ne 0$ for all $i$. Now, we follow the ideas of the proof of \cite[Proposition 1.c.6]{LiTzII}.
We consider in $Z$ the following semi-norm
$$
\trin a\trin \colon =\sup_{|b|\leq |a|} \|Sb\|_X.
$$

Note that if $\trin a\trin=0$, since $|a_i e_i|\leq |a|$ we have
that $\|S(a_i e_i)\|=0$ and  then $a_i=0$ for all $i$. Thus, $\trin
\cdot \trin$ is indeed a norm. Let $\widetilde{Z}$ be the completion
of $Z$ with this norm and let $\iota\colon (Z, \|\cdot \|_Z) \to
(\widetilde Z, \trin\cdot\trin)$ be the natural inclusion. Note that
if $|b|\leq |a|$  then $\|Sb\|_X\leq
\|S\|\|b\|\leq \|S\|\|a\|$ and $\iota  \colon Z\to \widetilde{Z}$ is
bounded with $\|\iota\|\le \|S\|$. It is easy to see that $\tilde Z$
is a solid sequence space, its canonical unit
vectors being $\tilde e_i:=\iota(e_i)$.

The subspace $J(X)$ is complemented in $Z$ and isomorphic to $X$, then $J(X)$ does not contain a subspace isomorphic to $c_0$. Since our construction of $\widetilde Z$ coincides with that of \cite[Proposition 1.c.6]{LiTzII}, we are in conditions to ensure that $\widetilde Z$ is order continuous. In this case, the unit vectors form a basis for $\widetilde{Z}$. Indeed, since $\widetilde Z$ is solid, it is enough to show that every $a\in \widetilde Z$ with $a\ge 0$ belongs to $\overline{\gen}\{\tilde e_i\}$. But for such $a$, the sequence $a-\sum_{i=1}^Na_i \tilde e_i$ decreases to 0 in order and, by order continuity, in norm. We have then seen that $\widetilde Z$ is a Schauder sequence space with an unconditional basis.

Now, we may consider $\theta$ the restriction of $\iota$ to $J(X)$ and put $\widetilde Y=\theta J(X)$ obtaining a subspace of $\widetilde Z$ isomorphic to $X$. We have the following commutative diagram:

$$
\xymatrix{
&   & Z \ar[dl]_S   \ar[r]^\iota
                 & \widetilde Z \ar[d]^{\widetilde S}     \\
& X  \ar[r]^J    & JX \ar[r]^\theta  \ar@{^{(}->}[u]  & \widetilde Y }
$$
where $\widetilde{S}$ is defined on $\iota(Z)$ by  $\widetilde S \iota = \theta JS$ and is then extended by continuity and density to $\widetilde Z$ (the continuity of $\widetilde{S}$ on $\iota(Z)$ follows from the definition of $\trin\cdot\trin$).

We claim that $((x'_i),S\theta^{-1}\widetilde S)$ is a Banach frame for $X$ with respect to $\widetilde Z$. If that is the case, since $S\theta^{-1}\widetilde S(e_i)= S(e_i)$ and $\widetilde Z$ is a Schauder sequence space, we would have the desired result. As $(\langle x'_i,x\rangle) \in Z$ and every sequence in $Z$ belongs to $\widetilde Z$, condition (a) of the definition of Banach frame holds. Also we have
$$
\|x\|=\|S(\langle x'_i,x\rangle)\|\le \trin (\langle x'_i,x\rangle)\trin =\trin \iota (Jx)\trin \le \|\iota\| \|J\|\|x\| \le \|S\|\|J\|\|x\|
$$
and the second condition is also satisfied. Finally, $S\theta^{-1}\widetilde S(\langle x'_i,x\rangle) = S\theta^{-1}\widetilde S( \iota (Jx))=x$ gives the third condition. \end{proof}

Note that Theorem~\ref{teo: BF lattice no c0 es DA incond} applies, for example,
to reflexive Banach spaces, or Banach spaces with finite cotype. In particular, any unconditional Banach frame for a subspace of $L_p$  or of a Lorentz function space $L_{p,q}$ ($1\le p<\infty$) has automatically a reconstruction formula (see~\cite{Cr} for cotype of Lorentz function spaces $L_{p,q}$). Analogous results can be obtained for many Lorentz or Orlicz functions spaces, the cotype of which are widely studied.
\bigskip

Theorem 1.c.7 in \cite{LiTzII} asserts that given $Y$ a
complemented subspace of a Banach lattice, $Y$ is reflexive if and
only if no subspace of $Y$ is isomorphic to $c_0$ or to $\ell_1$.
From this result and the comments previous to Remark~\ref{rem: subsp and proj} we have, in
particular, that if $X$ has an unconditional Banach frame, then $X$ is reflexive if and only if $X$ does not contain a
copy of $c_0$ or $\ell_1$. On the other hand, if a Banach space
admits an atomic decomposition which is both shrinking and boundedly
complete, then  it is reflexive \cite[Proposition 2.4]{CaLa}. The
converse holds under the additional assumption that the reflexive
space admits an unconditional atomic decomposition \cite[Theorem
2.5]{CaLa}. We combine and rephrase these results
as:

\begin{remark}\label{coro: shr-acot compl equiv c0 y ell1}
Let $X$ be a Banach space which admits an unconditional atomic decomposition $((x'_i),(x_i))$.
Then, the following are equivalent:
\begin{enumerate}
\item[(i)] $((x'_i),(x_i))$ is shrinking and boundedly complete,
\item[(ii)] $X$ does not contain a copy of $c_0$ or $\ell_1$,
\item[(iii)] $X$ is reflexive.
\end{enumerate}
\end{remark}

Our goal now is to show that, just as in the Schauder basis context,
we can split the equivalence (i)$\Leftrightarrow$(ii) in the
previous remark into two independent results. Note that in the previous remark and also the following results, the words ``atomic decomposition'' can be readily replaced by ``Schauder frame''. As a consequence, we improve some results in~\cite{Liu}. First we have:

\begin{theorem}\label{teo: shr y ell1}
Let $X$ be a Banach space which admits an unconditional atomic decomposition $((x'_i),(x_i))$.
Then, $((x'_i),(x_i))$ is shrinking if and only if  $X$ does not contain a copy of $\ell_1$.
\end{theorem}

\begin{proof}
Suppose $((x'_i),(x_i))$ is shrinking, by Proposition~\ref{prop: DA
shr y serie conv} $((x_i),(x'_i))$ is an atomic decomposition for
$X'$ with respect to some Banach sequence space, in particular, $X'$
is separable. Then, $X$ contains no subspace isomorphic to $\ell_1$.

Conversely, suppose $X$ does not admit a copy of $\ell_1$. Let $X_d$
be the canonical associated solid space respect to $((x'_i),(x_i))$
with synthesis operator $S$. Note that $(X_d)'=X_d^\times$, the dual
of K\"othe, then it is a sequence space and we may consider the
coordinate functions $(e_i'')$. Let $J\colon X\to X_d$ be the
analysis operator. Since $J'S'=id_{X'}$, $((S'' e''_i),J')$ is a
Banach frame for $X'$  with respect to $(X_d)'$. Now, $S'' e''_i=S
e_i$. Indeed, for all $x'\in X'$ we have that
$$
\langle S'' e''_i,x'\rangle=\langle e''_i,S'x'\rangle=\langle
S'x',e_i\rangle=\langle x',Se_i\rangle.
$$
Thus, $((S e_i),J')$ is a Banach frame for $X'$ respect to some solid space.
Since $X$ contains no copy of  $\ell_1$, $X'$ contains no copy of $c_0$ (\cite[Proposition~2.e.8]{LiTzI}). Therefore, by Theorem~\ref{teo: BF lattice no c0 es DA incond},
$((Se_i),(J' e'_i))$ is an unconditional atomic decomposition for $X'$. Moreover, we have $x_i=Se_i$ and $x'_i=J' e'_i$, then we obtain the reconstruction formula
$x'=\sum_i \langle x',x_i\rangle x'_i$ for all $x'\in X'$. Finally, by Proposition~\ref{prop: DA shr y serie conv}, $((x'_i),(x_i))$ is shrinking.
\end{proof}

Regarding the containment of $c_0$, we obtain:

\begin{theorem}\label{teo: acot completa y c0}
Let $X$ be a Banach space which admits an unconditional atomic decomposition $((x'_i),(x_i))$.
Then, $((x'_i),(x_i))$ is boundedly complete if and only if  $X$ does not contain a copy of $c_0$.
\end{theorem}

\begin{proof}
Suppose that $X$ contains a copy of $c_0$. Then, $X$ being
separable, by Sobczyc's theorem (\cite[Theorem 2.5.8]{AlKa}),
there exists a projection $P\colon X\to X$ such that $P(X)$ is
isomorphic to $c_0$. If
$((x'_i),(x_i))$ were boundedly complete, then, by Remark~\ref{obs: DA
acot completa proy}, $((P'x'_i),(Px_i))$ should be a boundedly
complete atomic decomposition for $P(X)$. This fact contradicts
Remark~\ref{obs: DA acot compl y bidual}.

Conversely, suppose that $((x'_i),(x_i))$ is not boundedly complete. Then, there exists $x''\in X''$ such that $\sum_{i=1}^\infty \langle x'',x'_i\rangle x_i$ is nonconvergent. In other words, we can find $\delta>0$ and two sequences of positive integers $(p_i), (q_i)$, so that
$p_1<q_1<p_2<q_2<p_3<q_3<\cdots$ and
$$
\big\|\sum_{i=p_j}^{q_j} \langle x'',x'_i\rangle x_i\big\|\geq
\delta,\quad\mbox{for all}\ j.
$$

Take $y_j=\sum_{i=p_j}^{q_j} \langle x'',x'_i\rangle x_i$. We will show that $c_0$ is embeddable in $X$. First, let us see that there exists $c>0$ so that for any positive integer $N$ and any choice of scalars $a_1,\ldots,a_N$, we have:
\begin{equation}\label{eq: bloques para acot completa}
\big\|\sum_{j=1}^N a_j y_j\big\|\leq c\max_{1\leq j \leq N}
|a_j|.
\end{equation}

Fix $\varepsilon>0$, by Goldstine's lemma, given $N\in\mathbb N$, we can find $x_N\in X$ such that
$\|x_N\|\leq \|x''\|$ and
$$
\big\|\sum_{j=1}^N a_j y_j\big\|=\big\|\sum_{i=1}^M b_i \langle
x'',x'_i\rangle x_i\big\| \leq \big\|\sum_{i=1}^M b_i \langle
x_N,x'_i\rangle x_i\big\| + \varepsilon,
$$
where $b_i$ is $a_j$ for some $j$ or $0$. Now, by Remark~\ref{obs: DA incond y
bilineal} we have that
$$
\big\|\sum_{i=1}^M b_i \langle x_N,x'_i\rangle x_i\big\|\leq
\|b\|_\infty \|x_N\|\leq \|a\|_\infty \|x''\|.
$$
Thus, we obtain~\eqref{eq: bloques para acot completa} for $c=
\|x''\|$. Since $\|y_j\|> \delta$, by the Bessaga-Pelczynski
theorem, it only remains to show that $y_j\stackrel{w}{\rightarrow}
0$. If this were not the case, passing to a subsequence if
necessary, we may assume that there exists $x'_0\in X'$ so that
$|\langle x'_0,y_j\rangle|\geq 1$ for all $j$. Now, take
$b_j=sign(\langle x'_0,y_j\rangle)$,
$$
N\leq \sum_{i=1}^N |\langle x'_0,y_j\rangle|=\big|\sum_{i=1}^N b_j
\langle x'_0,y_j\rangle\big|\leq \|x'_0\|\big\|\sum_{i=1}^N b_j
y_j\big\|\leq c\|x'_0\|\quad\mbox{for all}\ N,
$$
which is a contradiction. Then, $y_j\stackrel{w}{\rightarrow} 0$ and
we have that $X$ admits a copy of $c_0$ by a direct application of
\cite[Theorem 14.2]{AlBu}.
\end{proof}

As a consequence, we have the converse of
Remark~\ref{obs: DA acot compl y bidual} for spaces with unconditional atomic decompositions.
Indeed, if $X$ is complemented in its bidual, it cannot contain $c_0$
(since, by Sobczyc theorem this copy would be complemented, and this would provide a projection
from $\ell_\infty$ to $c_0$). So we have:

\begin{corollary}\label{coro: DA acot compl y bidual 2}
Let $X$ be a Banach space with a unconditional atomic decomposition.
Then, $X$ is complemented in its bidual if and only if the atomic decomposition
is boundedly complete (if and only if $X$ does not contain a copy of $c_0$).
\end{corollary}

For Banach frames, we have an analogous result, which follows from
Corollary~\ref{coro: DA acot compl y bidual 2} and Theorem \ref{teo:
BF lattice no c0 es DA incond}:

\begin{corollary}\label{coro: BF solid y bidual}
Suppose $X$ admits an unconditional Banach frame.
Then, $X$ is complemented in its bidual if and only if $X$ does not contain a copy of $c_0$.
\end{corollary}

\subsection*{Acknowledgements} The authors wish to express their gratitude to Professor Hans G. Feichtinger for the suggestions he made to improve this manuscript and also for providing them with additional references on Banach frames.

\end{document}